\documentclass[11pt,a4paper,openany]{amsart}
\usepackage{mathrsfs}

\usepackage{amsmath,amsfonts,ams,verbatim}
\usepackage{latexsym}
\usepackage{amssymb,leftidx}
\usepackage{extarrows}
\usepackage{overpic}
\usepackage{color}
\usepackage{epsfig}
\usepackage{subfigure}

\newcommand\scc{\mathrm{sc}}
\newcommand\bb{\mathrm{b}}

\numberwithin{equation}{section}
\newtheorem{proposition}{Proposition}[section]
\newtheorem{definition}{Definition}[section]
\newtheorem{lemma}{Lemma}[section]
\newtheorem{theorem}{Theorem}[section]

\newtheorem{remark}{Remark}[section]

\begin{document}
\title[Global-in-time estimates]{Global-in-time Strichartz estimates for Schr\"odinger on scattering manifolds}

\author{Junyong Zhang}
\address{Department of Mathematics, Beijing Institute of Technology, Beijing
100081, China and Department of Mathematics, Stanford
University,  USA}
\email{zhang\_junyong@bit.edu.cn}

\author{Jiqiang Zheng}
\address{Universit\'e C\^{o}te d'Azur, CNRS, LJAD, France}
\email{zhengjiqiang@gmail.com, zheng@unice.fr}

\maketitle


\begin{abstract}
We study the global-in-time Strichartz estimates for the Schr\"odinger equation on a class of scattering manifolds $X^{\circ}$. Let $\mathcal{L}_V=\Delta_g+V$ where $\Delta_g$ is the Beltrami-Laplace
operator on the scattering manifold and $V$ is a real potential function on this setting.
We first extend the global-in-time Strichartz estimate in Hassell-Zhang \cite{HZ} on the requirement of $V(z)=O(\langle z\rangle^{-3})$ to $O(\langle z\rangle^{-2})$
and secondly generalize the result to the scattering manifold with a mild trapped set as well as Bouclet-Mizutani\cite{BM} but with a potential.
We also obtain a global-in-time local smoothing estimate on this geometry setting.

\end{abstract}

\begin{center}
 \begin{minipage}{120mm}
 \end{minipage}
 \end{center}

\maketitle 

\section{Introduction}

We continue the investigations carried out in \cite{HZ} about the global-in-time Strichartz estimates on a class of scattering manifold introduced by Melrose \cite{Melrose}.
There are too many work devoted to the study of Strichartz inequalities to cite all here and we focus on the Strichartz estimates on the scattering, i.e. asymptotically conic,  setting.

Let $(X^\circ, g)$ be the scattering manifold of dimension $n\geq3$ and let $\Delta_g$ be the nonnegative Laplace operator on $X^\circ$,
assume that $X^\circ$ is nontrapping,  Hassell-Tao-Wunsch \cite{HTW} established the
local in time Strichartz inequalities
\begin{equation}\label{local-S}
\|e^{it\Delta_g}u_0\|_{L^q_tL^r_z([0,1]\times X^\circ)}\leq
C\|u_0\|_{L^2(X^\circ)}
\end{equation}
where
$(q, r)$ is an \emph{admissible pair}, i.e.
\begin{equation}\label{admissible}
2\leq q,r\leq\infty, \quad 2/q+n/r=n/2,\quad (q,r,n)\neq(2,\infty,2).
\end{equation}
This result was improved to global-in-time and generalized to $\mathcal{L}_V=\Delta_g+V$ where $V$ with suitably regular and decaying at infinity in Hassell-Zhang \cite{HZ}, and
we record the result here
\begin{equation}\label{l-S}
\begin{split}
\|e^{it\mathcal{L}_V}u_0\|_{L^q_t(\R;L^r_z(X^\circ))}\leq C\|u_0\|_{L^2_z(X^\circ)}
\end{split}
\end{equation}
and
\begin{equation}\label{in-S}
\begin{split}
\big\|\int_0^te^{i(t-s)\mathcal{L}_V}F(s)ds\big\|_{L^q_t(\R;L^r_z(X^\circ))}\leq C\|F\|_{L^{\tilde{q}'}_t(\R;L^{\tilde{r}'}_z(X^\circ))}.
\end{split}
\end{equation}
where $(q,r), (\tilde{q},\tilde{r})$ are any admissible pairs. It is known that the Strichartz estimate must have some loss of derivative when the manifold has some trapped
geodesic flow; see \cite{BGT}. However  Burq-Guillarmou-Hassell \cite{BGH} proved the local-in-time Strichartz estimate without loss
 on the scattering manifold with a trapped set  which is hyperbolic and of sufficiently small fractal dimension.
If a set is trapped in this sense i.e. in \cite{BGH}, we say the trapped set is a mild trapped set.
In a very recent work Bouclet-Mizutani\cite{BM}, the authors generalized the above Strichartz estimates to gobal-in-time one (except the endpoint estimate i.e $q=2$) on a scattering manifold allowing with mild trapped
geodesic (in sense of \cite{BGH}) but without any potential. In this paper, we aim to first extend the global-in-time Strichartz estimate in Hassell-Zhang \cite{HZ} on the requirement of $V(z)=O(\langle z\rangle^{-3})$ to $O(\langle z\rangle^{-2})$ and also secondly generalize the result to the scattering manifold with a mild trapped set as well as \cite{BGH, BM} but with a potential.
We remark that the global-in-time estimate is more delicate than the local-in-time one since one need to understand the boundedness of operators, e.g. resolvent operator, both at
high frequency and low energy. \vspace{0.2cm}

Our problem lies in the scattering geometric
setting, which is the same as in \cite{HTW, HZ, VW}. Let $(X^\circ,g)$ be a complete non-compact Riemannian manifold of
dimension $n\geq2$ with one end diffeomorphic to $(0,\infty)\times
Y$ and let $X$ be a its compactified manifold with boundary $Y=\partial X$. A function $x$ is said to be a \emph{boundary defining function}
for $X$ if $x$ is positive smooth function on $X$ such that $\partial X=\{x=0\}$ and $dx\neq 0$ on $\partial X$. Following Melrose \cite{Melrose}, we say a Riemannian metric $g$ on $X$ is
a \emph{scattering metric} if we can write $g$ in a collar neighborhood of $\partial X$ for some choice of boundary defining function $x$ as follows
\begin{equation}\label{sc-metric}
g=\frac{\mathrm{d}x^2}{x^4}+\frac{h(x)}{x^2}
\end{equation}
where $h\in C^\infty(Sym^2(T^*X))$ is a smooth family of metrics on $\partial X$. The manifold $(X^\circ,g)$ is called a asymptotically conic manifold or scattering manifold if
$g$ is a scattering metric.  Near the boundary $\partial X$, we
use the local coordinates $(x,y)$ on $X$ where $y=(y_1,\cdots,y_{n-1})$ is the local coordinates on $Y=\partial
X$, and  use the coordinate $z=(z_1,\cdots,z_n)$ when away
from $\partial X$.
We say the manifold $X$ is
non-trapping if every geodesic
$z(s)$ in $X$ reaches $Y$ as $s\rightarrow\pm\infty$. The function $r:=1/x$ near $x=0$ can be thought of as a
``radial" variable near infinity and $y = (y_1, \dots, y_{n-1})$ can be regarded as  $n-1$
``angular" variables.
\vspace{0.2cm}

Let $(X^\circ,g)$ be a scattering manifold and let $dv=\sqrt{g}dz$ be the measure  induced by the metric $g$, we define the complex
Hilbert space $L^2(X^\circ)$ is given by the inner product
\begin{equation*}
\langle f,
g\rangle_{L^2(X^\circ)}=\int_{X^\circ}f(z)\overline{g(z)} dv
\end{equation*}
Define
$\Delta_g=\nabla^*\nabla$ to be the positive Laplace-Beltrami operator on
$X$;  Consider the Sch\"odinger operator
\begin{equation}
\mathcal{L}_V:=\Delta_g+V(z)
\end{equation}
where the potential $V$ is a real function on $X$ such that
\begin{equation}\label{A1}
V\in C^\infty(X),~ V(x,y)=O(x^2) ~\text{as}~x\rightarrow0.
\end{equation}
and $V_0:=(x^{-2}V)|_{\partial X}$ satisfies
\begin{equation}\label{A2}
\Delta_{h}+(n-2)^2/4+V_0>0~\text{on}~L^2(\partial X, h(0)).
\end{equation}
Here $\Delta_{h}$ is the positive Laplacian with respect to
the metric $h(0)$ and \eqref{A2} is meant in the strict sense that the bottom of the spectrum of the operator is \emph{strictly}  positive. Note that this means that $V_0=0$ is allowed for $n \geq 3$ but not $n=2$.
We assume that
\begin{equation}\label{A3} \mathcal{L}_V :=\Delta_g+V~\text{has no nonpositive eigenvalues or zero-resonance}.
\end{equation}
Compared with the assumption on the operator $\mathcal{L}_V$ in Hassell-Zhang \cite{HZ}, most assumptions are same except $V(x,y)=O(x^2)$ instead of $V(x,y)=O(x^3)$ as $x\to 0$.
The assumption allows some negative potential $V(x,y)>-cx^2$ with some small positive constant $c<(n-2)^2/4$. Note that if $V\geq 0$, zero resonance does not exist.
Under these assumptions, we can use the results of \cite{GHS1} but not \cite{GHS2}. \vspace{0.2cm}

As mentioned in \cite[Remark 3.7]{HZ}, if $V \in x^2 C^\infty(X)$ and $V_0 := x^{-2} V |_{\partial X}$ takes values in the range $(-(n-2)^2/4, 0)$, then from \cite[Corollary 1.5]{GHS1} we see that the $L^1 \to L^\infty$ norm of the propagator is at least a constant times $t^{-(\nu_0 + 1)}$ as $t \to \infty$, where $\nu_0^2$ is the smallest eigenvalue of $\Delta_{h} + V_0 + (n-2)^2/4$. Under the above assumption on the range of $V_0$,
we see that $\nu_0 <(n-2)/2$. This implies that the dispersive estimate (1-12) in \cite{HZ} will no longer be valid  as $|t-s| \to \infty$, hence we can not obtain dispersive estimate and then
use Keel-Tao's abstract method to obtain the full set of Strichartz estimate as \cite{HZ} did. However the global-in-time Strichartz estimate
still can be derived from the usual Rodnianski-Schlag method \cite{RS}, for example the Strichartz estimate established in \cite{BPST} for negative inverse-square potentials on $\R^n$.
We will first use method of Vasy-Wunsch \cite{VW} and Bony-H\"afner \cite{BH} to obtain the resolvent estimate for $\mathcal{L}_V$ and then establish the Strichartz estimate via this idea. \vspace{0.2cm}

The geometry of manifold, which does not occurs on Euclidean space, has great affect on the establishment of the global-in-time Strichartz estimate, for example the conjugated points mentioned
in \cite{HW, HZ} and the trapped geodesics \cite{B, Doi}. The non-trapping assumption in  \cite{HZ} is not necessary to obtain the Strichartz estimate, it seems the scattering manifold with a mild trapped set in sense of
\cite{BGH} where the local-in-time Strichartz estimate were established without loss. Since the trapping only influences sensitive at high frequency and however the high frequency is corresponding to the short time
dynamic behavior of Sch\"odinger operator, the low frequency estimates (corresponding to long-time) in \cite{HZ} and the strategy for high frequency in \cite{BGH} will allow us to obtain the
global-in-time Strichartz estimate. As well as \cite{BM}, we also provide an alternative proof based on the results in \cite{HZ} and \cite{BGH}.\vspace{0.2cm}

The main purpose of this paper is to prove
the following results.

\begin{theorem}\label{thm} Let $(X^\circ,g)$ be a scattering manifold of dimension
$n\geq3$. Let $\mathcal{L}_V=\Delta_g+V$ satisfy
\eqref{A1}, \eqref{A2} and \eqref{A3}.

(i) if $(X^\circ,g)$ is nontrapping, then for all admissible pair $(q,r)\in [2,\infty]^2$ satisfying \eqref{admissible}, it holds
\begin{equation}\label{Str-est-n}
\|e^{it\mathcal{L}_V}u_0\|_{L^q_tL^r_z(\mathbb{R}\times X^\circ)}\leq
C\|u_0\|_{L^2(X^\circ)}.
\end{equation}

(ii) if $(X^\circ,g)$ has a hyperbolic trapped set satisfying the assumptions of \cite{BGH} and $V(z)=O(\langle z\rangle^{-2-\epsilon})$ for any small $\epsilon>0$, then
\begin{equation}\label{Str-est-t}
\|e^{it\mathcal{L}_V}u_0\|_{L^q_tL^r_z(\mathbb{R}\times X^\circ)}\leq
C\|u_0\|_{L^2(X^\circ)}
\end{equation}
holds for  all admissible pair $(q,r)$ with $q>2$ satisfying \eqref{admissible}.

\end{theorem}

\begin{remark} Compared with our previous result \cite{HZ}, the first result is new for the requirement on the decay of the potential $V$. The second result is same as \cite{BM} but with a short range potential.
\end{remark}

Our next result is about the inhomogeneous Strichartz estimate including the double endpoint case $q=\tilde{q}=2$ under the non-trapping assumption.
\begin{theorem}\label{thm-in} Let $(X^\circ,g)$ and $\mathcal{L}_V=\Delta_g+V$  be in the first case of  Theorem \ref{thm}, i.e. the non-trapping case,
 then the inhomogeneous Strichartz estimate holds for all admissible pair $(q,r), (\tilde{q}, \tilde{r})\in [2,\infty]^2$, i.e. satisfying \eqref{admissible}
\begin{equation}\label{Str-est-in}
\left\|\int_0^t  e^{i(t-s)\mathcal{L}_V}F(s) ds\right\|_{L^q_tL^r_z(\mathbb{R}\times X^\circ)}\leq
C\|F\|_{L^{\tilde{q}'}_tL^{\tilde{r}'}_z(\mathbb{R}\times X^\circ)}.
\end{equation}
\end{theorem}
The non-double-endpoint inhomogeneous Strichartz estimate, namely $q>\tilde{q}'$, can be obtained by the the Christ-Kiselev lemma \cite{CK} as usual.
The double-endpoint estimate can be established through the ideas of \cite{Da} and \cite{BM1}.\vspace{0.2cm}

Now we introduce some notation. We use $A\lesssim B$ to denote
$A\leq CB$ for some large constant C which may vary from line to
line and depend on various parameters, and similarly we use $A\ll B$
to denote $A\leq C^{-1} B$. We employ $A\sim B$ when $A\lesssim
B\lesssim A$. If the constant $C$ depends on a special parameter
other than the above, we shall denote it explicitly by subscripts.
For instance, $C_\epsilon$ should be understood as a positive
constant not only depending on $p, q, n$, and $M$, but also on
$\epsilon$. Throughout this paper, pairs of conjugate indices are
written as $p, p'$, where $\frac{1}p+\frac1{p'}=1$ with $1\leq
p\leq\infty$.
\vspace{0.2cm}

This paper is organized as follows: In Section 2, we recall the b-geometry,  sc-geometry and the results in \cite{VW}.
Section 3 is devoted to the proof  of the first part result and we prove the Strichartz estimate on the manifold with mild trapped set in Section 4. Section 5 proves the
inhomogeneous Strichartz estimate in Theorem \ref{thm-in}. In the final appendix section, we record the
Kato smooth theorem and the Mourre theory for convenience.\vspace{0.2cm}

{\bf Acknowledgments:}\quad  The authors would like to thank Andrew Hassell and Andras Vasy for their
helpful discussions and encouragement. The first author is grateful for the hospitality of the Australian National University and
Stanford University, where the project was initiated and finished. The first author thanks Prof. Haruya Mizutani for his invaluable comments about the inhomogeneous estimate.
J. Zhang was supported by National Natural
Science Foundation of China (11401024), and China Scholarship Council and J. Zheng was supported by the European
Research Council, ERC-2014-CoG, project number 646650 Singwave.\vspace{0.2cm}

\section{Preliminaries: boundary and scattering geometry}

Let $X$ (the compactification of $X^\circ$) be a complete compact manifold with boundary $Y=\partial X$, we briefly recall the basic definitions of the boundary (b-) and scattering (sc-) structures on our setting $X$.
Recall the function $x$ is a boundary defining function, let $\dot{\mathcal{C}}^\infty(X)=\cap_k x^k\mathcal{C}^\infty(X)$  (called the set of Schwartz functions)  dentoe
smooth functions on $X$ vanishing to infinite order at $\partial X$. The dual of $\dot{\mathcal{C}}^\infty(X)$ is tempered distributional densities $\mathcal{C}^{-\infty}(X;\Omega X)$; and
tempered distributions $\mathcal{C}^{-\infty}(X)$ are elements of the dual Schwartz densities $\dot{\mathcal{C}}^{\infty}(X;\Omega X)$.

\begin{definition}[ b-vector fields and scattering vector fields] Define $\mathcal{V}_b(X)$ to be the Lie algebra of all smooth vector fields on $X$ which are tangent to the boundary
and the Lie algebra of  scattering vector fields is defined as $\mathcal{V}_{\scc}(X)=x\mathcal{V}_b(X)$.
\end{definition}

More precisely, these b-vector field can be realized as the sections of a
vector bundle $\leftidx{^{\bb}}{TX}$, called the b-tangent bundle.
That means $\mathcal{V}_{\bb}(X)=\mathcal{C}^\infty(X;\leftidx{^{\bb}}{TX})$,
i.e. $\mathcal{V}_{\bb}(X)$ is a space of sections of
$\leftidx{^{\bb}}{TX}$ the b-tangent bundle over $X$.  Using above notation in which
$x$ is the boundary defining function of $X$ and $y$ are coordinates
in $\partial X$, we have
\begin{equation*}\mathcal{V}_{\bb}(X)=\begin{cases}\mathcal{V}, ~\text{i.e. all $\mathcal{C}^\infty$-vector
fields}, &\text{in the interior} ~X;\\
\text{span}\{x\partial_x, \partial_{y_1}\cdots,
\partial_{y_{n-1}}\}, & \text{near the boundary} ~\partial X.
\end{cases}
\end{equation*}
We denote by $\text{Diff}_{\bb}^*(X)$ the `enveloping algebra' of $\mathcal{V}_{\bb}(X)$, meaning the ring of differential operator on $\mathcal{C}^\infty(X)$ generated by $\mathcal{V}_{\bb}(X)$ and $\mathcal{C}^\infty(X)$.
In particular, near the boundary $\partial X$, the $m$-order scattering differential operator is given by
\begin{equation*}\text{Diff}_{\bb}^m(X)=\Big\{A: A=\sum_{j+|\alpha|\leq
m}a_{j\alpha}(x,y)(x\partial_x)^j(\partial_y)^\alpha,
a_{j\alpha}\in \mathcal{C}^\infty(X) \Big\}.
\end{equation*}
The dual bundles of $\leftidx{^{\bb}}{TX}$ is $\leftidx{^{\bb}}{T^*X}$ with local base $$\frac{dx}{x},dy_1,\cdots, dy_{n-1}.$$
The b-density bundle is $$\nu_\bb=|\frac{dx}{x}dy_1\cdots dy_{n-1}|.$$ About the sc-vector field, we similarly have
\begin{equation*}\mathcal{V}_{\scc}(X)=\begin{cases}\mathcal{V}, ~\text{i.e. all $\mathcal{C}^\infty$-vector
fields}, &\text{in the interior} ~X;\\
\text{span}\{x^2\partial_x, x\partial_{y_1}\cdots,
x\partial_{y_{n-1}}\}, & \text{near the boundary} ~\partial X.
\end{cases}
\end{equation*}
and define the $m$-order scattering differential operator
\begin{equation*}\text{Diff}_{\text{sc}}^m(X)=\Big\{A: A=\sum_{j+|\alpha|\leq
m}a_{j\alpha}(x,y)(x^2\partial_x)^j(x\partial_y)^\alpha,
a_{j\alpha}\in \mathcal{C}^\infty(X) \Big\}.
\end{equation*}
Locally near the boundary, in the coordinate $(x,y)$, we have
\begin{equation*}\leftidx{^{\scc}}{T^*X}=
\text{span}\Big\{\frac{\mathrm{d}x}{x^2},
\frac{\mathrm{d}{y}}{x}\Big\}=\text{span}\Big\{\mathrm{d}\big(\frac1x\big),
\frac{\mathrm{d}{y}}{x}\Big\}.
\end{equation*}
The sc-density bundle is $$\nu_\scc=\left|\frac{dx}{x^2}\frac{dy_1\cdots dy_{n-1}}{x^{n-1}}\right|.$$
Fixing a volume $\bb$- or $\scc$-density $\nu_\bb$ or $\nu_{\scc}$ on $X$, we respectively define
$L^2_{\bb}(X)$ or  $L^2_{\scc}(X)$ to be the metric space $L^2(X;\nu_\bb)$ and $L^2(X;\nu_\scc)$.  So $L^2_{\scc}(X)=x^{\frac n2}L^2_\bb(X)$. Without confusing, we
write $L^2_{\scc}(X)$ to $L^2(X)$. \vspace{0.2cm}

Our setting is the scattering manifold in which the metric $g$ is a scattering metric, then the Laplacian $\Delta_g\in \text{Diff}_{\scc}^2(X)$
and we can write $\Delta_g=x^2 P_b$ where $P_b\in \text{Diff}_{\bb}^2(X)$. More explicitly, in local coordinates $(x,y)$ on a collar neighborhood of $\partial X$,
we have by Melrose \cite[Proof of Lemma 3]{Melrose}
$$P_b=D_x x^2 D_x+i(n-1)xD_x+\Delta_h+x^\rho R, \quad\rho>0,~ R\in\text{Diff}_{\bb}^2(X).$$

We recall  \cite[Proposition 4.5]{VW} here
\begin{lemma}\label{lem: bound} Let $Q\in x^{1+s}\mathrm{Diff}^1_b(X)$, $0\leq s<1/2$, the following estimate holds for all $f\in L_\scc^2(X)$ and $\mathrm{Im} \sigma\neq 0$
\begin{equation}
\|Q(2^j\mathcal{L}_V-\sigma)^{-1}f\|_{L_\scc^2(X)}\lesssim 2^{-(1+s)j/2}\left(\frac{\sigma}{\mathrm{Im} \sigma}\right)^{\frac{1+s}2} |\mathrm{Im}\sigma|^{-(1-s)/2}\|f\|_{L_\scc^2(X)}.
\end{equation}
\end{lemma}

 \begin{lemma}\label{lem:mou} Let $A=-\frac12\left((\phi xD_x)+(\phi xD_x)^*\right)$ where $\phi\in \mathcal{C}_c^\infty(X)$ supported in a collar
neighborhood of $\partial X$ and be identically $1$ near $\partial X$. Define $A_j=\Phi(2^j\mathcal{L}_V) A\Phi(2^j\mathcal{L}_V)$ where $\Phi \in\mathcal{C}_c^\infty([\frac12,2])$.
The we have the following uniform estimates for $j\geq 0$
\begin{gather}
\|[A_j, (2^j\mathcal{L}_V)^{1/2}]\|_{L^2\to L^2}\lesssim 1,\\
\|[A_j, [A_j, (2^j\mathcal{L}_V)^{1/2}]]\|_{L^2\to L^2}\lesssim 1,\\
\||A_j|^\mu x^\mu\|_{L^2\to L^2}\lesssim 2^{-\mu j/2}, \quad \mu\in [0,1],\\
\|\langle A_j\rangle^\mu \Phi(2^j\mathcal{L}_V) x^\mu\|_{L^2\to L^2} \lesssim 2^{-\mu j/2}, \quad \mu\in [0,1].
\end{gather}
and the Mourre estimate
\begin{equation}
\chi_{I}(2^j\mathcal{L}_V)\frac{i}2[(2^j\mathcal{L}_V)^{1/2}, A_j]\chi_{I}(2^j\mathcal{L}_V)\geq C\chi_{I}(2^j\mathcal{L}_V).
\end{equation}
where $\chi_I$ is the characteristic function of the compact interval $I\subset(0,\infty)$.
\end{lemma}

\begin{proof}This directly follows from  the formulas (6.1)-(6.4) and Theorem 5.3 in \cite{VW}. Even though the operator $\mathcal{L}_V$ considered here
is a bit different from the operator $P=\Delta_g+V$ stated in Theorem 5.3 of \cite{VW} where $V\geq0$ and $V=O(x^{2+\epsilon})$ for $\epsilon>0$ as $x\to 0$,
but their argument can go through with minor modifications for $V>-c x^2$ such that $\Delta_{h} + V_0 + (n-2)^2/4$ is positive, for example $c<(n-2)^2/4$.
The minor modifications also have been indicated in footnotes in \cite{VW}.
\end{proof}

\section{Strichartz estimate for Schr\"odinger equation with $\mathcal{L}_V$}

In this section, we prove the first part of Theorem \ref{thm} under the non-trapping condition. The proof is based on the result in \cite{HZ}.

Consider $\mathcal{L}_0=\Delta_g$ without the potential, the operator $\mathcal{L}_0$ falls into the class of operator $\mathbf{H}$ considered in \cite{HZ}. Let us briefly recall the main strategy there.
To avoid the conjugate points, we  microlocalized (in phase space) propagators
$U_j(t)$ by
\begin{equation}\label{Uiti}
\begin{split}
U_j(t) = \int_0^\infty e^{it\lambda^2} Q_j(\lambda)
dE_{\sqrt{\mathbf{H}}}(\lambda), \quad 1 \leq j \leq N,
\end{split}\end{equation}
where $Q_j(\lambda)$ is a partition of the identity operator in
$L^2(X^\circ)$.
Then the operator $U_j(t) U_j(s)^*$
is given
\begin{equation}
U_j(t) U_j(s)^* =  \int e^{i(t-s)\lambda^2}  Q_j(\lambda)
dE_{\sqrt{\mathbf{H}}}(\lambda) Q_j(\lambda)^*.
\label{Uiti2}\end{equation}
We proved a uniform estimate on $\| U_j(t) \|_{L^2 \to L^2}$ in \cite[Section 5]{HZ} and dispersive estimate  in \cite[Section 6]{HZ}  for $U_j(t) U_j(s)^*$, the
homogeneous Strichartz estimate for $e^{it\mathbf{H}}$ finally was obtained by Keel-Tao's formalism \cite{KT} to each $U_j$ and summing over $j$. For
the endpoint inhomogeneous estimate, we required additional argument to obtain dispersive estimate on $U_i(t) U_j(s)^*$ for $i \neq j$ and the Keel-Tao's argument showed
the desirable endpoint inhomogeneous Strichartz estimate.\vspace{0.2cm}

As mentioned in \cite{KT}, the Strichartz estimates obtained by the abstract Keel-Tao's formalism can be sharped in Lorentz space norm $L^{r,2}(X)$.
More precisely, we have
\begin{lemma}\label{Stri-Lor} For any admissible pairs $(q,r)$ and $(\tilde{q},\tilde{r})$, the following Strichartz inequalities hold
\begin{equation}
\begin{split}
\|e^{it\mathcal{L}_0}u_0\|_{L^q_t(\R;L^{r,2}_z(X^\circ))}\leq C\|u_0\|_{L^2_z(X^\circ)}
\end{split}
\end{equation}
and
\begin{equation}
\begin{split}
\big\|\int_0^te^{i(t-s)\mathcal{L}_0}F(s)ds \big\|_{L^q_t(\R;L^{r,2}_z(X^\circ))}\leq C\|F\|_{L^{\tilde{q}'}_t(\R;L^{\tilde{r}',2}_z(X^\circ))}
\end{split}
\end{equation}
where $L^{r,2}$ is the Lorentz space on $X^\circ$.
\end{lemma}
Following the method of Rodnianski-Schlag \cite{RS}, see also \cite{BPST}, the Strichartz estimate is a consequence of
the global-in-time local smoothing
\begin{equation}\label{est:local-smooth}
\begin{split}
\|\langle
z\rangle^{-1}e^{it\mathcal{L}_V}u_0\|_{L^{2}_t(\R;L^{2}_z(X^\circ))}\leq
C\|u_0\|_{L^2}.
\end{split}
\end{equation}
Indeed by Duhamel's formula, we have for any admissible pair $(q,r)$ with $r\geq2$
\begin{equation*}
\begin{split}
&\|e^{it\mathcal{L}_V}u_0\|_{L^q_tL^r_z}\lesssim \|e^{it\mathcal{L}_V}u_0\|_{L^q_tL^{r,2}_z}\\&\lesssim
\|e^{it\mathcal{L}_0}u_0\|_{L^q_tL^{r,2}_z}+\big\|\int_0^t
e^{i(t-s)\mathcal{L}_0}V(z)e^{is\mathcal{L}_V}u_0 ds\big\|_{L^q_tL^{r,2}_z}
\\&\lesssim
\|u_0\|_{L^2_z}+\|V(z)e^{is\mathcal{L}_V}u_0\|_{L^2_tL^{\frac{2n}{n+2},2}_z}
\\&\lesssim \|u_0\|_{L^2_z}+\|\langle z\rangle
V(z)\|_{L^{n,\infty}}\|\langle
z\rangle^{-1}e^{is\mathcal{L}_V}u_0\|_{L^{2}_sL^{2}_z}\lesssim \|u_0\|_{L^2_z}.
\end{split}
\end{equation*}
Now it suffices to show \eqref{est:local-smooth}, by Kato's smoothing theorem (e.g. Theorem \ref{kato-sm} below), which follows from the resolvent estimate
\begin{proposition}\label{est:resovent} There exists $\epsilon_0>0$ and a constant $C$ such that for all $\sigma\in\C$ satisfying $|\mathrm{Im}\sigma|\leq \epsilon_0$ we have

(i) if $(X^\circ,g)$ is nontrapping
\begin{equation}\label{resolvent-est1}
\begin{split}
\|\langle z\rangle^{-1}(\mathcal{L}_V-\sigma)^{-1}\langle
z\rangle^{-1}\|_{L^2\to L^2}\leq \frac{C}{\sqrt{1+|\sigma|}};
\end{split}
\end{equation}

(ii) if $(X^\circ,g)$ has the trapped set satisfying the assumptions of \cite{BGH}
\begin{equation}\label{resolvent-est2}
\begin{split}
\|\langle z\rangle^{-1}(\mathcal{L}_V-\sigma)^{-1}\langle
z\rangle^{-1}\|_{L^2\to L^2}\leq \frac{C}{\sqrt{1+\frac{|\sigma|}{\log(1+|\sigma|)}}}.
\end{split}
\end{equation}

\end{proposition}
\begin{remark}This resolvent is strong enough to obtain \eqref{est:local-smooth}. Indeed by Kato's smoothing, we only need
$$
\|\langle z\rangle^{-1}(\mathcal{L}_V-\sigma)^{-1}\langle
z\rangle^{-1}\|_{L^2\to L^2}\leq C.$$
The stronger statement at $|\sigma|\gg1$ actually gains $(Id+\mathcal{L}_V)^{1/4}$-derivatives than \eqref{est:local-smooth}.
When $\mathrm{Re} \sigma \in I\subset (0,\infty)$ in a compact set, this result is due to Melrose \cite{Melrose} and while $\mathrm{Re} \sigma \gg1$,
this is a direct consequence of Vasy-Zworski\cite{VZ} for non-trapping and  Nonnenmacher-Zworski\cite{NZ}, Datchev\cite{D} and Datchev-Vasy \cite{DV} for mild trapped case where the weight is $\langle
z\rangle^{-1/2-\delta}$ with $\delta>0$.

\end{remark}

\begin{proof} Let $\sigma=\lambda+i\epsilon$ with $\lambda,\epsilon\in\R$ with $|\epsilon|\leq \epsilon_0$. Since $\mathcal{L}_V$ is a non-negative operator,
the spectrum $\sigma(\mathcal{L}_V)\subset (0,\infty)$. By the functional calculus, it is easy to prove \eqref{resolvent-est1} and \eqref{resolvent-est2} when $\lambda<0$. We only consider $\lambda\geq0$.
If $\lambda \in I$ with $I \subset (0,\infty)$ being a compact set, this result is essentially due to Melrose \cite{Melrose} and while $\lambda\gg1$, the behavior of $\mathcal{L}_V-\lambda$ as $\lambda\to \infty$
is equivalent to the semiclassical operator $h^2\mathcal{L}_V-1$ as $h=\lambda^{-1/2}\to 0$,
this is a direct consequence of Vasy-Zworski\cite{VZ} and Datchev-Vasy \cite{DV}where the weight can be sharped to $\langle
z\rangle^{-1/2-\delta}$ with $\delta>0$. Hence it suffices to consider resolvent estimate at the low frequency $\lambda\lesssim 1$ which is given by Proposition \ref{l-res-est} below.

\end{proof}

\begin{proposition}[Resolvent estimate at low energy]\label{l-res-est}  Let $\mathcal{L}_V=\Delta_g+V$,
we have the following estimates, uniformly in $\sigma\in\C\setminus [0,+\infty)$ and $|\sigma|<1$,
\begin{equation}\label{res-est'}
\|\langle z\rangle^{-1}(\mathcal{L}_V-\sigma)^{-1} \langle z\rangle^{-1}\|\lesssim 1.
\end{equation}

\end{proposition}

\begin{remark} Bouclet-Royer\cite{BR} showed this when $V=0$. On the asymptotically Euclidean space, Bony-H\"afner \cite{BH} proved the resolvent estimate
\begin{equation}
\|\langle z\rangle^{-\alpha}(\mathcal{L}_V-\sigma)^{-1} \langle z\rangle^{-\beta}\|\lesssim 1
\end{equation}
provided $\alpha,\beta>1/2$ and $\alpha+\beta>2$.
\end{remark}

{\bf Proof.} On our setting, this result is essentially due to Vasy-Wunsch \cite{VW} and Bony-H\"afner \cite{BH}. We prove the resolvent estimate by following the method of \cite{BH} and using
Mourre estimate established in \cite{VW} on this setting. Let  $\varphi\in C_0^\infty([1/2,2])$ take value in $[0,1]$ and $\varphi_0=0, s<2$ satisfy that
$$\varphi_0(s)+\sum_{j\geq0}\varphi_j(s)=1,\quad \varphi_j(s)=\varphi(2^js).$$
Since $0$ is not an eigenvalue of $\mathcal{L}_V$,  we have
$$\mathrm{Id}=\varphi_0(\mathcal{L}_V)+\sum_{j\geq0}\varphi(2^j\mathcal{L}_V)$$
where we define the operators $\varphi_j(\mathcal{L}_V)$ $(j\in\N)$ via the Helffer-Sj\"ostrand formula
\begin{equation}
\begin{split}
\varphi_j(\mathcal{L}_V)&=\frac{1}{2\pi}\int_{\C} \bar{\partial}\tilde{\varphi_j}(\sigma) (\mathcal{L}_V-\sigma)^{-1} L(d\sigma)
\\&=\frac{1}{2\pi}\lim_{\epsilon\to0}\int_{|\mathrm{Im}\sigma|\ge\epsilon} \bar{\partial}\tilde{\varphi_j}(\sigma) (\mathcal{L}_V-\sigma)^{-1} L(d\sigma)
\end{split}
\end{equation} here $\tilde{\varphi_j}$ is an almost analytic extension of $\varphi_j$. Hence we write
\begin{equation}
\begin{split}
&\langle z\rangle^{-1}(\mathcal{L}_V-\sigma)^{-1} \langle z\rangle^{-1}\\
&=\langle z\rangle^{-1}\left( \varphi_0(\mathcal{L}_V)+\sum_{j\geq0}\varphi_j(\mathcal{L}_V)\right)(\mathcal{L}_V-\sigma)^{-1} \langle z\rangle^{-1}\\
&=\langle z\rangle^{-1}\varphi_0(\mathcal{L}_V)(\mathcal{L}_V-\sigma)^{-1} \langle z\rangle^{-1}
+\sum_{j\geq0} \langle z\rangle^{-1}\varphi_j(\mathcal{L}_V)(\mathcal{L}_V-\sigma)^{-1} \langle z\rangle^{-1} .
\end{split}
\end{equation}
Since $|\sigma|<1$ and $\varphi_0$ is supported in $[2,\infty)$, the functional calculus shows
\begin{equation}
\begin{split}
\|\langle z\rangle^{-1}\varphi_0(\mathcal{L}_V)(\mathcal{L}_V-\sigma)^{-1} \langle z\rangle^{-1} \|_{L^2\to L^2}\lesssim 1.
\end{split}
\end{equation}
Let $\phi\in C_c^\infty([\frac14,4])$ and $\phi=1$ on the support of $\varphi$ and let $\Phi\in C_c^\infty([\frac1{16},16])$ and $\Phi=1$ on $[\frac18,8]$. Define $\phi_j(s)=\phi(2^js)$ and
$\Phi_j(s)=\Phi(2^js)$. The norm $\|\cdot \|$ briefly denotes the  $\|\cdot \|_{L^2\to L^2}$, we have
\begin{equation*}
\begin{split}
&\|\langle z\rangle^{-1}\varphi_j(\mathcal{L}_V)(\mathcal{L}_V-\sigma)^{-1} \langle z\rangle^{-1} \|
=\|\langle z\rangle^{-1}\varphi_j(\mathcal{L}_V)\Phi_j(\mathcal{L}_V)(\mathcal{L}_V-\sigma)^{-1} \phi_j(\mathcal{L}_V)\langle z\rangle^{-1} \|\\
&\lesssim \|\langle z\rangle^{-1}\varphi_j(\mathcal{L}_V)\langle z\rangle\|\|\langle z\rangle^{-1}
\Phi_j(\mathcal{L}_V)(\mathcal{L}_V-\sigma)^{-1} \langle z\rangle^{-1}\| \|\langle z\rangle\phi_j(\mathcal{L}_V)\langle z\rangle^{-1} \|.
\end{split}
\end{equation*}
To prove \eqref{res-est'}, as the last operator is same to the adjoints of the first one, it suffices to show
\begin{lemma} For $j\geq 1$, there exists a constant $C$ independent of $j$ such that
\begin{equation}\label{est:b1}
 \|\langle z\rangle^{-1}\varphi_j(\mathcal{L}_V)\langle z\rangle\|\leq C
 \end{equation}
 and
 \begin{equation}\label{est:b2}
 \|\langle z\rangle^{-1}
\Phi_j(\mathcal{L}_V)(\mathcal{L}_V-\sigma)^{-1} \langle z\rangle^{-1}\| \leq C
 \end{equation}
uniformly in $\sigma\in\C\setminus \R$ and $|\sigma|<1$.

\end{lemma}

\begin{proof} We begin to prove the first one. Recall $\varphi_j(\mathcal{L}_V)=\varphi(2^j\mathcal{L}_V)$, since the spectrum theory implies $\varphi(2^j\mathcal{L}_V)$ is bounded on $L^2$ with norm
$\sup_{s}|\varphi(s)|\leq 1$ and
\begin{equation}
\langle z\rangle^{-1}\varphi(2^j\mathcal{L}_V)\langle z\rangle= \langle z\rangle^{-1}[\varphi(2^j\mathcal{L}_V), \langle z\rangle]+\varphi(2^j\mathcal{L}_V),
 \end{equation}
 it suffices to estimate  $\langle z\rangle^{-1}[\varphi(2^j\mathcal{L}_V), \langle z\rangle]$ by functional calculus. Let $\tilde{\varphi}$ be a compactly supported almost analytic extension of $\varphi$ and let
 $R_j(\tau)$ denote the resolvent $$R_j(\tau)=(2^j \mathcal{L}_V-\tau)^{-1},$$ we have by Helffer-Sj\"ostrand formula
 \begin{equation}
 \begin{split}
 \langle z\rangle^{-1}[\varphi(2^j\mathcal{L}_V), \langle z\rangle]&=\frac1{2\pi}\int_{\C}\bar{\partial}\tilde{\varphi}(\tau)\langle z\rangle^{-1}[R_j(\tau), \langle z\rangle] L(d\tau)\\
 &=-\frac{2^j}{2\pi}\int_{\C}\bar{\partial}\tilde{\varphi}(\tau)\langle z\rangle^{-1}R_j(\tau) [\mathcal{L}_V, \langle z\rangle]R_j(\tau) L(d\tau)\\
 &=-\frac{2^j}{2\pi}\lim_{\epsilon\to0}\int_{|\mathrm{Im}\sigma|\geq\epsilon}\bar{\partial}\tilde{\varphi}(\tau)\langle z\rangle^{-1}R_j(\tau) [\mathcal{L}_V, \langle z\rangle]R_j(\tau) L(d\tau).
 \end{split}
 \end{equation}
Since $\Delta_g=x^2 P_\bb$,  we see $[\mathcal{L}_V, \langle z\rangle]=Q\in \mathrm{Diff}^1_{\scc}=x\mathrm{Diff}^1_{\bb}$, then
 from Lemma \ref{lem: bound}, we have
\begin{equation}
\|\langle z\rangle^{-1}\left(2^j\mathcal{L}_V-\tau\right)^{-1}\|_{L^2(X)\to L^2(X)}\lesssim 2^{-j/2} |\tau|^{1/2}|\mathrm{Im}\tau|^{-1}.
\end{equation}
and
\begin{equation}
\|Q\left(2^j\mathcal{L}_V-\tau\right)^{-1}\|_{L^2(X)\to L^2(X)}\lesssim 2^{-j/2} |\tau|^{1/2}|\mathrm{Im}\tau|^{-1}.
\end{equation}
Note that $|\bar{\partial}\tilde{\varphi}(\tau)|\leq C |\mathrm{Im}\tau|^N$ for any $N>0$ and $|\tau|\leq 2$ which is due to the compact support of $\varphi$,
then the above integral converges, and hence $\langle z\rangle^{-1}[\varphi(2^j\mathcal{L}_V), \langle z\rangle]$ is a bounded operator on $L^2$.

Next we prove \eqref{est:b2}.
For $\chi\in C_c^\infty([\frac1{200}, \frac1{100}])$
\begin{equation*}
  \begin{split}
& \|\langle z\rangle^{-1}
\Phi(2^j\mathcal{L}_V)(\mathcal{L}_V-\sigma)^{-1} \langle z\rangle^{-1}\|\\&\leq \|\langle z\rangle^{-1}
\Phi(2^j\mathcal{L}_V)(\mathcal{L}_V-\sigma)^{-1}\chi(2^j|\sigma|) \langle z\rangle^{-1}\|\\&\quad+\|\langle z\rangle^{-1}
\Phi(2^j\mathcal{L}_V)(\mathcal{L}_V-\sigma)^{-1}(1-\chi(2^j|\sigma|)) \langle z\rangle^{-1}\|.
\end{split}
 \end{equation*}
 Since the function $\Phi(2^j\tau)(\tau-\sigma)^{-1}$ is bounded when $\frac{1}{200}2^{-j}\leq |\sigma|\leq 2^{-j}/100$, the second operator is bounded by the spectrum theory.
 Hence it suffices to consider the $L^2$-boundedness of the operator $\|\langle z\rangle^{-1}
\Phi(2^j\mathcal{L}_V)(\mathcal{L}_V-\sigma)^{-1}\chi(2^j |\sigma|) \langle z\rangle^{-1}\|$, which is same as to the boundedness of
$$\|\langle z\rangle^{-1}
\Phi(2^j\mathcal{L}_V)(\mathcal{L}_V-2^{-j}\sigma)^{-1} \langle z\rangle^{-1}\|, \quad |\sigma|\in [\frac1{200},\frac1{100}].$$
Let $A=-\frac12\left((\phi xD_x)+(\phi xD_x)^*\right)$ where $\phi\in \mathcal{C}_c^\infty(X)$ supported in a collar
neighborhood of $\partial X$ and be identically $1$ near $\partial X$. Define $A_j=\Phi(2^j\mathcal{L}_V) A \Phi(2^j\mathcal{L}_V) $. Note that
  \begin{equation}
  \begin{split}
& \|\langle z\rangle^{-1}
\Phi(2^j\mathcal{L}_V)(\mathcal{L}_V-2^{-j}\sigma)^{-1} \langle z\rangle^{-1}\|\\
&\leq \|\langle z\rangle^{-1}
\tilde{\Phi}(2^j\mathcal{L}_V) (\sqrt{\mathcal{L}_V}+(2^{-j}\sigma)^{1/2})^{-1} \langle z\rangle\|\\& \|\langle z\rangle^{-1}
\Phi(2^j\mathcal{L}_V) A_j\|\| A_j^{-1}(\sqrt{\mathcal{L}_V}-(2^{-j}\sigma)^{1/2})^{-1}A_j^{-1}\| \| A_j \tilde{\Phi}(2^j\mathcal{L}_V) \langle z\rangle^{-1}\|
\\
&\leq 2^{-j/2}
\| A_j^{-1}(\sqrt{\mathcal{L}_V}-(2^{-j}\sigma)^{1/2})^{-1}A_j^{-1}\| \\
&\leq
\| A_j^{-1}((2^j\mathcal{L}_V)^{1/2}-\sqrt{\sigma})^{-1}A_j^{-1}\|
\end{split}
 \end{equation}
 where we use
 \begin{equation}
 \|\langle A_j\rangle \Phi(2^j \mathcal{L}_V) x\|\lesssim 2^{-j/2}, \quad \mu\in [0,1].
 \end{equation}
 and
  \begin{equation}
  \|\langle z\rangle^{-1}
\tilde{\Phi}(2^j\mathcal{L}_V) (\sqrt{\mathcal{L}_V}+(2^{-j}\sigma)^{1/2})^{-1} \langle z\rangle\|\leq \frac{2^{j/2}}{|\sigma|}.
 \end{equation}
From Lemma \ref{lem:mou} which verifies the condition of Theorem \ref{thm:mourre}, one has by the Mourre theory
 \begin{equation}
  \begin{split}
  \| A_j^{-1}((2^j\mathcal{L}_V)^{1/2}-\sqrt{\sigma})^{-1}A_j^{-1}\|\lesssim 1.
\end{split}
 \end{equation}
 Hence we prove \eqref{est:b2}.

\end{proof}

\section{Strichartz estimates on the setting with a mild trapped set}

In this section, we prove the second part of Theorem \ref{thm}. The argument is based on \cite{HZ} and \cite{BGH}.
We always assume $0<h\leq h_0\ll1$ and $\varphi\in\mathcal{C}_c^\infty((1,2))$ and  let $\varphi_0(\lambda)=\sum_{j\leq 0}\varphi (2^{-j}\lambda)$  in this section. \vspace{0.2cm}

We first give a global-in-time local smoothing estimate localized at high frequency.
\begin{proposition}[global-in-time local smoothing]\label{local-l} Let $\chi\in \mathcal{C}_0^\infty(X)$ and $\varphi\in\mathcal{C}_c^\infty((1,2))$, then we have
\begin{equation}
\|\chi e^{it\Delta_g}\varphi(h^2\Delta_g)u_0\|_{L^2(\R;L^2(X))}\leq a(h)\|u_0\|_{L^2}
\end{equation}
where $a(h)\leq Ch^{1/2}$ if $\chi$ is supported outside the trapped set $K$ and otherwise $a(h)\leq C(h\log h)^{1/2}$.
\end{proposition}

\begin{proof}This result was due to \cite{BGH} even though it was stated as a local-in-time version but the argument works for this global-in-time estimate.
The proof directly follows from the resolvent estimates due to Vasy-Zworski \cite{VZ} for non-trapping case and Nonnenmacher-Zworski\cite{NZ}, Datchev\cite{D}, Datchev-Vasy\cite{DV}
for mild trapped case
\begin{equation}
\|\chi(h^2\Delta_g-(1\pm i\epsilon))^{-1}\chi\|_{L^2(X)\to L^2(X)}\lesssim \frac{|\log h|}h,\quad 0<h<h_0\ll1.
\end{equation}
\end{proof}

Next we prove the Strichartz estimate on the scattering manifold with a mild trapped set. We here
consider the case $V=0$ and let $\mathcal{L}_0=\Delta_g$.
\begin{proposition}[Global-in-time Strichartz estimate]\label{Str-loc} Let $\chi\in \mathcal{C}_c^\infty(X)$ such that $\chi=1$ on the trapped set. For the admissible pair $(q,r)$ satisfying \eqref{admissible} with $q>2$, we have the Strichartz estimate:

(i) Low frequency estimate
\begin{align}\label{l-str}
\|e^{it\mathcal{L}_0}\varphi_0(\mathcal{L}_0)u_0\|_{L^q(\R;L^r(X))}\leq C\|u_0\|_{L^2(X)};
\end{align}

(ii) High frequency estimate outside the trapped set
\begin{align}\label{out-trap}
\|(1-\chi) e^{it\mathcal{L}_0}\varphi(h^2\mathcal{L}_0)u_0\|_{L^q(\R;L^r(X))}\leq C\|\varphi(h^2\mathcal{L}_0)u_0\|_{L^2(X)};
\end{align}

(iii) High frequency estimate inside the trapped set
\begin{align}\label{h-in-trap}
\|\chi e^{it\mathcal{L}_0}\varphi(h^2\mathcal{L}_0)u_0\|_{L^q(\R;L^r(X))}\leq C\|\varphi(h^2\mathcal{L}_0) u_0\|_{L^2(X)}.
\end{align}

\end{proposition}

\begin{proof} We first consider the low frequency estimate \eqref{l-str} which follows the same argument of the non-trapping case, since the trapped set only has
influence on the high frequency. We sketch here that the microlocalized spectrum measure in \cite{HZ} do not need the non-trapping condition for the low
frequency part and we also do not need the non-trapping assumption for resolvent estimate at low frequency. \vspace{0.2cm}

We secondly consider the estimate \eqref{out-trap} outside the trapped set.  Let $w=(1-\chi) e^{it\mathcal{L}_0}\varphi(h^2\mathcal{L}_0)u_0$ then $w$ solves
\begin{equation*}
i\partial_t w+ \mathcal{L}_0 w=-[\Delta_g, \chi]\varphi(h^2\mathcal{L}_0) u, \quad  w(0)= (1-\chi) \varphi(h^2\mathcal{L}_0) u_0(z).
\end{equation*}
By Duhamle formula, we have
\begin{equation*}
\begin{split}
w=&e^{it\mathcal{L}_0}(1-\chi) \varphi(h^2\mathcal{L}_0) u_0(z)-(1-\chi)\int_0^t e^{i(t-s)\mathcal{L}_0}[\Delta_g,\chi]\varphi(h^2\mathcal{L}_0) u(s)ds\\&
=w_{\mathrm{lin}}+w_{\mathrm{nonh}}.
\end{split}\end{equation*}
Since $\chi=1$ on the trapped set, we can regard the above equation as on a asymptotically conic manifold without trapped set. Hence we can apply the
Strichartz estimate \eqref{Str-est-n} and the Christ-Kiselev lemma \cite{CK} to obtain for $q>2$
\begin{equation*}
\begin{split}
&\|w(t,z)\|_{L^q(\R;L^r(X^\circ))}\\&\lesssim \|\varphi(h^2\mathcal{L}_0) u_0\|_{L^2(X^\circ)}+\big\|(1-\chi)\int_0^t e^{i(t-s)\mathcal{L}_0}[\Delta_g,\chi]\varphi(h^2\mathcal{L}_0) u(s)ds \big\|_{L^q(\R;L^r(X^\circ))}.
\end{split}
\end{equation*}
Define $T=\chi e^{it\mathcal{L}_0}\varphi(h^2\mathcal{L}_0): L^2(X^\circ)\to L^2(\R;L^2(X^\circ))$, by Proposition \ref{local-l}, then we
have that $T$ is bounded with norm $h^{1/2}$. By the dual, its adjoint $T^*$ is also bounded by $h^{1/2}$,
 where $$T^*: L^2(\R;L^2(X^\circ))\to L^2, \quad T^* F=\int_{s\in\R}e^{-is\mathcal{L}_0} \varphi(h^2\mathcal{L}_0)\chi^* F(s)ds.$$
Define the operator
$$B: L^2(\R;L^2(X^\circ))\to L^q(\R;L^r(X^\circ)), \quad B F=\int_{s\in\R}(1-\chi) e^{i(t-s)\mathcal{L}_0} \varphi(h^2\mathcal{L}_0)\chi^* F(s)ds.$$
Hence by the Strichartz estimate on nontrapping
\begin{equation}
\begin{split}
&\|B F\|_{L^q(\R;L^r(X^\circ))}\\&=\big\|(1-\chi) e^{i t\mathcal{L}_0}\int_{s\in\R}e^{-is\mathcal{L}_0} \varphi(h^2\mathcal{L}_0)\chi^* F(s)ds\big\|_{L^q(\R;L^r(X^\circ))}\\
&\lesssim \big\|\int_{s\in\R}e^{-is\mathcal{L}_0} \varphi(h^2\mathcal{L}_0)\chi^* F(s)ds\big\|_{L^2(X^\circ)}\lesssim h^{1/2}\|F\|_{L^2(\R;L^2(X^\circ))}.
\end{split}
\end{equation}
By the Christ and Kiselev lemma \cite{CK}, we have done for $q>2$
\begin{equation}
\begin{split}
&\big\|(1-\chi)\int_0^t e^{i(t-s)\mathcal{L}_0}[\Delta_g,\chi]\varphi(h^2\mathcal{L}_0) u(s)ds \big\|_{L^q(\R;L^r(X^\circ))}\\&\lesssim h^{1/2} \|[\Delta_g,\chi]\varphi(h^2\mathcal{L}_0) u\|_{L^2(\R;L^2(X^\circ))}.
\end{split}
\end{equation}
Note that $[\Delta_g,\chi]$ is compact supported and losing one-derivative, by local smoothing Proposition \ref{local-l} which gains $h^{1/2}$, we obtain
\begin{equation}
\begin{split}
h^{1/2} \|[\Delta_g,\chi]\varphi(h^2\mathcal{L}_0) u\|_{L^2(\R;L^2(X^\circ))}\lesssim \|\varphi(h^2\mathcal{L}_0) u_0\|_{L^2(X^\circ)}.
\end{split}
\end{equation}

Finally we consider the high frequency inside the mild trapped set. The proof follows from the argument in \cite{BGH}. Let $a(h)=h |\log h|$ and let $v=\chi e^{it\mathcal{L}_0}\varphi(h^2\mathcal{L}_0)u_0$ then $v$ solves
\begin{equation*}
i\partial_t v+ \mathcal{L}_0 v=[\Delta_g, \chi]\varphi(h^2\mathcal{L}_0) u, \quad  v(0)= \chi \varphi(h^2\mathcal{L}_0) u_0(z).
\end{equation*}
Let $\phi(s)\in\mathcal{C}_0^\infty([-1,1])$ satisfy $\phi(0)=1$ and $\sum_{j\in\mathbb{Z}}\phi(s-j)=1$ and define $v_j=\phi(\frac t {a(h)}-j) v$, then $v=\sum_{j\in\mathbb{Z}} v_j$ and each $v_j$ supported on a time
interval of length $2a(h)$
satisfies
\begin{equation*}
\begin{cases}
i\partial_t v_j+ \mathcal{L}_0 v_j=F_j+G_j, \\ v_j(t)|_{t=(j-1)a(h)}=0.
\end{cases}
\end{equation*}
where $F_j=[i\partial_t, \phi(\frac t {a(h)}-j)]\chi \varphi(h^2\mathcal{L}_0)  u$ and $G_j=\phi(\frac t {a(h)}-j)[\Delta_g,\chi]\varphi(h^2\mathcal{L}_0)  u$.
Then by Duhamel's formula we have
\begin{equation*}
\begin{split}
v_j(t,z)=\int_{(j-1)a(h)}^t e^{i(t-s)\mathcal{L}_0}(F_j+G_j)(s)ds=\chi\int_{(j-1)a(h)}^t e^{i(t-s)\mathcal{L}_0}(F_j+G_j)(s)ds.
\end{split}\end{equation*}
Define
\begin{equation*}
\begin{split}
\tilde{v}^F_j(t,z)=\chi \int_{(j-1)a(h)}^{(j+1)a(h)} e^{i(t-s)\mathcal{L}_0}F_j(s)ds, \quad \tilde{v}^G_j(t,z)=\chi \int_{(j-1)a(h)}^{(j+1)a(h)} e^{i(t-s)\mathcal{L}_0}G_j(s)ds.
\end{split}\end{equation*}
Note the support of $\tilde{v}^F_j$ with respect to variable $t$ is contained in $[(j-1)a(h),(j+1)a(h)]$,
\begin{equation*}
\begin{split}
\|\tilde{v}^F_j(t,z)\|_{L^q(\R;L^r(X))}&=\left\|e^{it\mathcal{L}_0}\int_{(j-1)a(h)}^{(j+1)a(h)} e^{-is\mathcal{L}_0}F_j(s)ds\right\|_{L^q([(j-1)a(h),(j+1)a(h)];L^r(X^\circ))}
\\&\leq \left\|\int_{(j-1)a(h)}^{(j+1)a(h)} e^{-is\mathcal{L}_0}F_j(s)ds\right\|_{L^2(X^\circ)}\leq C a(h)^{1/2}\|F_j\|_{L^2(\R;L^2(X^\circ))},
\end{split}\end{equation*}
where we use the following in the first inequality and H\"older's inequality for the second inequality
\begin{lemma}\label{Strichartz-c}
We have the Strichartz estimate
\begin{equation*}
\begin{split}
\|\chi e^{it\mathcal{L}_0}\varphi(h^2\mathcal{L}_0)u_0\|_{L^q([t_0-a(h), t_0+a(h)];L^r(X^\circ))}\leq C\|u_0\|_{L^2(X^\circ)}, \quad \forall t_0\in\R.
\end{split}\end{equation*}
\end{lemma}
\begin{proof} This is direct consequence of \cite[Theorem 3.8]{BGH}.
\end{proof}

By the Christ-Kiselev lemma, we obtain that for $q>2$
\begin{equation*}
\begin{split}
\|v_j^F(t,z)\|_{L^q(\R;L^r(X^\circ))}\leq C a(h)^{1/2}\|F_j\|_{L^2(\R;L^2(X^\circ))}.
\end{split}\end{equation*}
Note that \begin{align*}
F_j=\frac{i} {a(h)} \phi'(\frac t {a(h)}-j)\chi \varphi(h^2\mathcal{L}_0)u,
\end{align*}
hence \begin{align*}
\sum_{j\in\Z}\|F_j\|^2_{L^2(\R;L^2(X^\circ))}\lesssim a(h)^{-2}\|\chi \varphi(h^2\mathcal{L}_0)u\|^2_{L^2(\R;L^2(X^\circ))}\lesssim a(h)^{-1}\|\varphi(h^2\mathcal{L}_0) u_0\|^2_{L^2}.
\end{align*}
On the other hand,
 \begin{equation*}
\begin{split}
\|\tilde{v}^G_j(t,z)\|_{L^q(\R;L^r(X))}&=\left\|e^{it\mathcal{L}_0}\int_{(j-1)a(h)}^{(j+1)a(h)} e^{-is\mathcal{L}_0}G_j(s)ds\right\|_{L^q([(j-1)a(h),(j+1)a(h)];L^r(X^\circ))}
\\&\leq \left\|\int_{(j-1)a(h)}^{(j+1)a(h)} e^{-is\mathcal{L}_0}G_j(s)ds\right\|_{L^2(X^\circ)}\leq C h^{1/2}\|G_j\|_{L^2(\R;L^2(X^\circ))},
\end{split}\end{equation*}
where we use Lemma \ref{Strichartz-c} again in the first inequality and while for the second inequality we use the duality estimate of the local smoothing with no trapped set since
$G_j$ vanishes at the trapped set $\pi(K)$. Similarly note that \begin{align*}
G_j=\phi(\frac t {a(h)}-j)(h^{-1}\nabla\chi+\Delta\chi) \varphi(h^2\mathcal{L}_0)u,
\end{align*}
Let $\chi_{\text{out}}=1$ on the support of $\nabla\chi$, then $\chi_{\text{out}} $ is supported outside the trapped set. Thus by Proposition \ref{local-l} \begin{align*}
\sum_{j\in\Z}\|G_j\|^2_{L^2(\R;L^2(X^\circ))}\lesssim h^{-2}\|\chi_{\mathrm{out}} \varphi(h^2\mathcal{L}_0)u\|^2_{L^2(\R;L^2(X^\circ))}\lesssim h^{-1}\|\varphi(h^2\mathcal{L}_0) u_0\|^2_{L^2}.
\end{align*}
Therefore by embedding $\ell^2(\Z)$ to $\ell^q(\Z)$ with $q\geq2$
\begin{equation}
\begin{split}
\|v\|^2_{L^q(\R;L^r(X^\circ))}&\lesssim \sum_{j\in\mathbb{Z}} \left(\|v^F_j\|^2_{L^q(\R;L^r(X^\circ))}+(\|v^G_j\|^2_{L^q(\R;L^r(X^\circ))}\right)
\\&\lesssim \sum_{j\in\mathbb{Z}} \left(a(h)\|F_j\|^2_{L^2(\R;L^2(X^\circ))}+h\|G_j\|^2_{L^2(\R;L^2(X^\circ))}\right)\\&\lesssim \|\varphi(h^2\mathcal{L}_0) u_0\|^2_{L^2}
\end{split}
\end{equation}
which shows \eqref{h-in-trap}.

\end{proof}

Using Proposition \ref{Str-loc} and the Littlewood-Paley estimate for $\mathcal{L}_0$ in \cite[Equation 1.4]{Bo} or \cite[Proposition 2.2]{zhang}, we sum the frequency to obtain
\eqref{Str-est-t} for $V=0$. Now by similarly argument as in last section, we obtain the Strichartz estimate from
the global-in-time local smoothing. For the mild trapped case, we do not know whether the Strichartz estimate can be sharped to Lorentz space hence
we need $V(z)=O(\langle z\rangle^{-2-\epsilon})$ such that $\langle z\rangle
V(z)\in L^n$. By Proposition \ref{est:resovent} and Kato's local smoothing, even though for the trapping case, we again have
\begin{equation*}
\begin{split}
\|\langle
z\rangle^{-1}e^{it\mathcal{L}_V}u_0\|_{L^{2}_t(\R;L^{2}_z(X^\circ))}\leq
C\|u_0\|_{L^2}.
\end{split}
\end{equation*}
Further by Duhamel's formula again, we have for any admissible pair $(q,r)$ with $q>2$
\begin{equation*}
\begin{split}
&\|e^{it\mathcal{L}_V}u_0\|_{L^q_tL^r_z}\\&\lesssim
\|e^{it\mathcal{L}_0}u_0\|_{L^q_tL^{r}_z}+\big\|\int_0^t
e^{i(t-s)\mathcal{L}_0}V(z)e^{is\mathcal{L}_V}u_0 ds\big\|_{L^q_tL^{r}_z}
\\&\lesssim
\|u_0\|_{L^2_z}+\|V(z)e^{is\mathcal{L}_V}u_0\|_{L^2_tL^{\frac{2n}{n+2}}_z}
\\&\lesssim \|u_0\|_{L^2_z}+\|\langle z\rangle
V(z)\|_{L^{n}}\|\langle
z\rangle^{-1}e^{is\mathcal{L}_V}u_0\|_{L^{2}_sL^{2}_z}\lesssim \|u_0\|_{L^2_z}.
\end{split}
\end{equation*}
Therefore we prove the second part of Theorem \ref{thm}.

\section{The inhomogeneous Strichartz estimate}
For $q>\tilde{q}'$, the inhomogeneous Strichartz estimate  can be proved by using the Christ-Kiselev lemma \cite{CK} and the above homogeneous Strichartz estimate of Theorem \ref{thm}.
To obtain the double-endpoint estimate, i.e. $q=\tilde{q}=2$,  we follow the methods of \cite{D} and \cite{BM1}.

Recall $\mathcal{L}_V=\Delta_g+V$ and $\mathcal{L}_0=\Delta_g$, define the operators
\begin{equation}
\begin{split}
\mathcal{N}_0 F(t)=\int_0^t e^{i(t-s)\mathcal{L}_0} F(s) ds,\quad \mathcal{N} F(t)=\int_0^t e^{i(t-s)\mathcal{L}_V} F(s) ds.
\end{split}
\end{equation}
Set $u(t)=e^{i(t-s)\mathcal{L}_V} F(s)$, then we can write
$$u(t)=e^{i(t-s)\mathcal{L}_0} F(s)-i\int_s^t e^{i(t-\tau)\mathcal{L}_0} \left(Ve^{-i(\tau-s)\mathcal{L}_V} F(s)\right) d\tau.$$
Integrating in $s\in[0,t]$, we have by Fuibni's formula
\begin{equation*}
\begin{split}
 \mathcal{N} F(t)&=\int_0^t e^{i(t-s)\mathcal{L}_V} F(s) ds\\&=\int_0^t e^{i(t-s)\mathcal{L}_0} F(s) ds-i\int_0^t \int_s^t e^{i(t-\tau)\mathcal{L}_0} \left(Ve^{-i(\tau-s)\mathcal{L}_V} F(s)\right) d\tau ds
 \\&=\mathcal{N}_0 F(t)-i\int_0^t \int_0^\tau e^{i(t-\tau)\mathcal{L}_0} \left(Ve^{-i(\tau-s)\mathcal{L}_V} F(s)\right)  ds d\tau
 \\&=\mathcal{N}_0 F(t)-i\int_0^t e^{i(t-\tau)\mathcal{L}_0} \left( V \int_0^\tau e^{-i(\tau-s)\mathcal{L}_V} F(s) ds \right)  d\tau
  \\&=\mathcal{N}_0 F(t)-i\mathcal{N}_0 \left( V (\mathcal{N} F) \right)(t).
\end{split}
\end{equation*}
Therefore \begin{equation}\label{NF}
\begin{split}
 \mathcal{N} F(t)=\mathcal{N}_0 F(t)-i\mathcal{N}_0 \left( V (\mathcal{N} F) \right)(t).
\end{split}
\end{equation}
On the other hand, we have by similar argument
$$\mathcal{N}_0 F(t)=\mathcal{N} F(t)+i\mathcal{N} \left( V (\mathcal{N}_0 F) \right)(t),$$ hence $$\mathcal{N} F(t)=\mathcal{N}_0 F-i\mathcal{N} \left( V (\mathcal{N}_0 F) \right)(t).$$
Plugging it into \eqref{NF},  we obtain
\begin{equation*}
\begin{split}
 \mathcal{N} F=\mathcal{N}_0 F-i\mathcal{N}_0 \left( V (\mathcal{N}_0 F) \right)-\mathcal{N}_0 \left( V (\mathcal{N} (V (\mathcal{N}_0 F))) \right),
\end{split}
\end{equation*}
that is
\begin{equation}
\begin{split}
 \mathcal{N} F=\mathcal{N}_0 F-i\left( \mathcal{N}_0 V \mathcal{N}_0  \right)F-\mathcal{N}_0 (V \mathcal{N} V) \mathcal{N}_0 F.
\end{split}
\end{equation}
For $q=2$, we need to estimate
\begin{equation*}
\begin{split}
&\| \mathcal{N} F\|_{L^2(\R;L^{\frac{2n}{n-2},2})}\\&
\leq \|\mathcal{N}_0 F\|_{L^2(\R;L^{\frac{2n}{n-2},2})}+\|\left(\mathcal{N}_0 V \mathcal{N}_0  \right)F\|_{L^2(\R;L^{\frac{2n}{n-2},2})}+\|\mathcal{N}_0 (V \mathcal{N} V) \mathcal{N}_0 F\|_{L^2(\R;L^{\frac{2n}{n-2},2})}.
\end{split}
\end{equation*}
By the Strichartz estimate in Lemma \ref{Stri-Lor},  we have
\begin{equation}\label{N-0}
\begin{split}
\|\mathcal{N}_0 F\|_{L^2(\R;L^{\frac{2n}{n-2},2})}\lesssim \|F\|_{L^2(\R;L^{\frac{2n}{n+2},2})}.
\end{split}
\end{equation}
Using the assumption \eqref{A1} of  potential $V$, then one has $V\in L^{\frac n2,\infty}$. Thus we obtain from the Strichartz estimate in Lemma \ref{Stri-Lor}
\begin{equation}\label{N-1}
\begin{split}
&\|\left(\mathcal{N}_0 V \mathcal{N}_0  \right)F\|_{L^2(\R;L^{\frac{2n}{n-2},2})}\\&\lesssim \|V \mathcal{N}_0 F\|_{L^2(\R;L^{\frac{2n}{n+2},2})}\lesssim \|\mathcal{N}_0 F\|_{L^2(\R;L^{\frac{2n}{n-2},2})}\lesssim \|F\|_{L^2(\R;L^{\frac{2n}{n+2},2})}.
\end{split}
\end{equation}
Using the assumption \eqref{A1} on $V$ again, then one has $\langle z\rangle V\in L^{n, \infty}$ and $\langle z\rangle^2 V\in L^{\infty}$. Similarly as above, we prove
\begin{equation}\label{N-2}
\begin{split}
&\|\mathcal{N}_0 (V \mathcal{N} V) \mathcal{N}_0 F\|_{L^2(\R;L^{\frac{2n}{n-2},2})}
\\&\lesssim \| (V \mathcal{N} V) \mathcal{N}_0 F\|_{L^2(\R;L^{\frac{2n}{n+2},2})}\leq \| (\langle z\rangle^{-1}\mathcal{N} \langle z\rangle^{-1}) (\langle z\rangle^2 V)\langle z\rangle^{-1} \mathcal{N}_0 F\|_{L^2(\R;L^{2})}\\&\lesssim \|\langle z\rangle^{-1} \mathcal{N}_0 F\|_{L^2(\R;L^{2})}
\lesssim \|\mathcal{N}_0 F\|_{L^2(\R;L^{\frac{2n}{n-2},2})}\lesssim \|F\|_{L^2(\R;L^{\frac{2n}{n+2},2})}.
\end{split}
\end{equation}
Here we use the following lemma about the local smooth estimate
\begin{lemma} Let $\mathcal{L}_V$ be in Theorem \ref{thm-in}, then we have
\begin{equation}\label{in-local-sm}
\|\langle z\rangle^{-1}\int_0^t e^{i(t-s)\mathcal{L}_V}\langle z\rangle^{-1} F ds\|_{L^2(\R;L^{2})}\leq C\|F\|_{L^2(\R;L^{2})}.
\end{equation}
\end{lemma}
\begin{proof} By Proposition \ref{est:resovent}, the operator $\langle z\rangle^{-1}$ is  $\mathcal{L}_V$-suppersmoothing operator.
Using D'ancona's result \cite{Da}, that is Theorem \ref{in-sm} in appendix,  and the density argument, we obtain \eqref{in-local-sm}.
\end{proof}
Finally collecting \eqref{N-0}, \eqref{N-1} and \eqref{N-2}, we show the double-endpoint estimate by Lorentz imbedding
\begin{equation*}
\left\|\int_0^t  e^{i(t-s)\mathcal{L}_V}F(s) ds\right\|_{L^2_tL^{\frac{2n}{n-2}}_z(\mathbb{R}\times X^\circ)}\leq
C\|F\|_{L^{2}_tL^{\frac{2n}{n+2}}_z(\mathbb{R}\times X^\circ)}.
\end{equation*}
Therefore we prove Theorem \ref{thm-in}.

\section{Appendix: Kato smoothing theorem and the Mourre  theory}
We record the Kato smooth theorem and the Mourre theory for convenience; see \cite{ABG} and \cite{BH1}.

Let $\mathcal{H}$, $\mathcal{H}_1$ be Hilbert spaces and $\mathbf{H}$ is a selfadjoint operator on $\mathcal{H}$ with domain $D(\mathbf{H})\subset \mathcal{H}$.
For $\sigma\in \C\setminus\R$, define the resolvent operator of $\mathbf{H}$ by  $R(\sigma) = (\mathbf{H}-\sigma)^{-1}$.

\begin{definition} \cite{KY}A closed operator $A:\mathcal{H}\to \mathcal{H}_1$ with dense domain $D(A)$ is said to be

(i) $\mathbf{H}$-smooth, with constant $a$, if there exists $\epsilon_0>0$ such that for every $\epsilon, \lambda\in\R$ with $|\epsilon|\leq \epsilon_0$ the
following uniform estimate holds
\begin{equation}
\left|\langle (2i)^{-1}(R(\lambda+i\epsilon)-R(\lambda-i\epsilon))A^* v, A^*v \rangle_{\mathcal{H}_1}\right|\leq a\|v\|_{\mathcal{H}_1},\quad \forall v\in D(A^*),
\end{equation}

(ii) $\mathbf{H}$-supersmooth, with constant $a$, one has that there exists $\epsilon_0>0$ such that for every $\epsilon, \lambda\in\R$ with $|\epsilon|\leq \epsilon_0$
\begin{equation}
\left|\langle R(\lambda+i\epsilon)A^* v, A^*v \rangle_{\mathcal{H}_1}\right|\leq a\|v\|_{\mathcal{H}_1},\quad \forall v\in D(A^*).
\end{equation}
\end{definition}

\begin{theorem}\label{kato-sm} Let $A:\mathcal{H}\to \mathcal{H}_1$ be a closed operator with dense domain $D(A)$.
Then $A$ is $\mathbf{H}$-smooth with constant $a$ if and only if, for any $v\in \mathcal{H}$, one has $e^{-it\mathbf{H}}v\in D(A)$ for almost every $t\in\R$ and
\begin{equation}
\|Ae^{-it\mathbf{H}}v\|_{L^2(\R;\mathcal{H}_1)}\leq 2\sqrt{a}\|v\|_{\mathcal{H}}.
\end{equation}
\end{theorem}
\begin{proof} See \cite[Lemma 3.6, Theorem 5.1]{Kato} or see \cite[Theorem XIII.25] {RS}.
\end{proof}

We next recall the result of \cite[Theorem 2.3]{Da} which is used to prove the endpoint inhomogeneous Strichartz estimate.

 A step function $h(t):\R\to\mathcal{H}$ is a measurable function of bounded support taking a finite number of values;
 measurability and integrals of Hilbert-valued functions are in the sense of Bochner.

\begin{theorem}\label{in-sm} Let $A:\mathcal{H}\to \mathcal{H}_1$ be a closed operator with dense domain $D(A)$.
Assume $A$ is $\mathbf{H}$-supersmooth with constant $a$. Then for almost every $t\in\R$ and any $v\in \mathcal{H}$ one has $e^{-it\mathbf{H}}v\in D(A)$;
Furthermore, for any step function $h(t):\R\to D(A^*)$, $Ae^{-i(t-s)\mathbf{H}}A^*h(s)$ is Bochner integrable in s over $[0,t]$ (or $[t,0]$) and satisfies, for all $|\epsilon|\leq \epsilon_0$, the estimate
\begin{equation}\label{non-h}
\left\|e^{-\epsilon t}\int_0^t Ae^{-i(t-s)\mathbf{H}}A^* h(s)ds\right\|_{L^2(\R;\mathcal{H}_1)}\leq 2a\|e^{-\epsilon t}h(t)\|_{L^2(\R;\mathcal{H})}
\end{equation}
Conversely, if \eqref{non-h} holds, then $A$ is $\mathbf{H}$-supersmooth with constant $2a$.
\end{theorem}
\begin{proof} See \cite[Theorem 2.3]{Da}.
\end{proof}

Next we record the abstract Mourre  theory which implies the limit absorbing theorem and thus obtain the resolvent estimate.

\begin{definition} Let $(A, D(A))$ and $(\mathbf{H}, D(\mathbf{H}))$ be self-adjoint operator on a separable Hilbert space $\mathcal{H}$.
The operator $\mathbf{H}$ is of class $C^k(A)$ for $k>0$, if there exists $z\in \C\setminus\sigma(\mathbf{H})$ such that
$$\R\ni t \longrightarrow e^{itA}(\mathbf{H}-z)^{-1}e^{-it A},$$
is $C^k$ for the strong topology of $\mathcal{L}(\mathcal{H})$.
\end{definition}

Let $\mathbf{H}\in C^1(A)$ and $I\subset \sigma(\mathbf{H})$ be an open interval. We assume that $A$ and $\mathbf{H}$ satisfy a Mourre estimate on $I$
\begin{equation}\label{mourre-est}
\chi_I(\mathbf{H}) i[\mathbf{H}, A]\chi_I(\mathbf{H})\geq \delta \chi_I(\mathbf{H}), ~\text{for some}~\delta>0.
\end{equation}
Define the multi-commutators a  $ad_A^j B$ inductively by
$$ad_A^{j+1} B=[A, ad_A^j B], \quad ad_A^0 B=B.$$

\begin{theorem}[Limiting absorption principle]\label{thm:mourre} Let $\mathbf{H}\in C^2(A)$ be such that $ad_A^j \mathbf{H}$ $(j=1,2)$, are bounded on $\mathcal{H}$.
Assume furthermore \eqref{mourre-est}. Then, for all closed intervals $J\subset I$ and $\mu>1/2$, there exists $C_{J,\mu}>0$ such that
\begin{equation}\label{res-est}
\sup_{\Re z\in J, \Im z\neq 0}\|\langle A\rangle^{-\mu}(\mathbf{H}-z)^{-1}\langle A\rangle^{-\mu}\|\leq C_{J,\mu}.
\end{equation}

\end{theorem}

\begin{center}

\end{center}

\end{document}